\documentclass{amsart}

\usepackage[mathscr]{eucal}
\usepackage{amsmath,amssymb}
\usepackage{graphics}
\usepackage{multirow}
\usepackage{hyperref}
\usepackage{enumitem}
\usepackage{tikz}
\usetikzlibrary{matrix}
\usepackage[active]{srcltx}

\newtheorem{thm}{Theorem}[section]
\newtheorem{THM}{Theorem}

\newtheorem{COR}[THM]{Corollary}

\newtheorem{prop}[thm]{Proposition}

\newtheorem{lemma}[thm]{Lemma}

\theoremstyle{definition}

\newtheorem{remark}[thm]{Remark}
\newtheorem{example}[thm]{Example}

\DeclareMathOperator{\Res}{Res}

\DeclareMathOperator{\Hom}{Hom}

\DeclareMathOperator{\Aut}{Aut}

\DeclareMathOperator{\sing}{sing}

\DeclareMathOperator{\Aff}{Aff(\mathbb C)}
\DeclareMathOperator{\PSL}{PSL(2,\mathbb C)}
\DeclareMathOperator{\Affn}{Aff(V)}

\DeclareMathOperator{\Diff}{Diff}

\DeclareMathOperator{\GL}{GL}

\def\C{\mathbb C}

\def\Z{\mathbb Z}
\def\F{\mathcal F}

\def\D{\mathcal D}
\def\G{\mathcal G}
\def\H{\mathcal H}

\def\calN{N}

\def\fnabla{\widehat{\nabla}}
\def \Pu{\mathbb P^1}

\input xy
\xyoption{all}

\begin{document}

\title[Transversely affine foliations]
{Transversely affine foliations on projective manifolds}
\author[G. Cousin and J.V. Pereira]
{Ga\"el  COUSIN and Jorge Vit\'{o}rio PEREIRA}
\address{\newline
IMPA, Estrada Dona Castorina, 110, Horto, Rio de Janeiro,
Brasil}\email{jvp@impa.br, gael@impa.br}

\thanks{Jorge Vit\'orio Pereira is partially supported by CNPq-Brazil. During preparation of this work, Ga\"el Cousin was a fellow of CNPq-Brazil}
\begin{abstract}
We describe the structure of singular transversely affine foliations of codimension one on projective manifolds with zero first Betti number.
Our result can be rephrased as a theorem on rank two reducible flat meromorphic connections.
\end{abstract}

\maketitle

\setcounter{tocdepth}{1}
 \sloppy

\section{Introduction}

In this paper we study holomorphic foliations of codimension one on projective manifolds which are  singular transversely affine  in  the sense of  \cite{MR1432053}. These are natural generalizations of (smooth) transversely affine  foliations of codimension one as defined in \cite{MR1120547}. The classical definition is weakened at two points: the transverse structure is defined only on the complement of a divisor (but  extends meromorphically through this divisor); and its developing map  is not necessarily a submersion. A precise definition is given in Section \ref{S:definition}.

A Theorem due to Singer \cite{MR1062869} says that the class of singular transversely affine foliations of codimension one, roughly speaking,  coincides with the class of codimension one foliations which admit first integrals that can be obtained by iteration of the following three operations: resolution of algebraic equations, exponentiation, and integration of closed $1$-forms. More precisely, there exists a Liouvillian extension (cf. loc. cit. for a   definition) of the field of rational functions of the ambient manifold containing a non constant first integral for the foliation.

Our main result describes the structure of singular  transversely affine foliations of codimension one on a projective manifold $X$
rather precisely, at least under the assumption $h^1(X,\mathbb C) =0$.

\begin{THM}\label{THM:A}
Let $X$ be a projective manifold with $h^1(X,\mathbb C)=0$ and let $\mathcal F$ be a singular transversely affine foliation of codimension one on $X$.   Then at least one of following assertions holds true.
\begin{enumerate}
\item There exists a generically finite Galois morphism $p:Y\to X$ such that
$p^*\mathcal F$ is defined by a closed rational $1$-form.
\item There exists a transversely affine Ricatti foliation $\mathcal R$ on a surface $S$ and
a rational map $p:X \dashrightarrow S$ such that $p^* \mathcal R = \mathcal F$.
\end{enumerate}
\end{THM}

Our proof does not use the hypothesis on the topology of $X$ when the  transverse affine structure is regular (the connection on $N\mathcal F$ has at most logarithmic singularities); or when  the monodromy of the transverse affine structure is Zariski dense.

We do not know if the hypothesis on the topology is necessary in general. The result as stated above probably holds also on compact K\"ahler manifolds, but at some key points we used results that are only available in the algebraic category. We do know that the result does not hold for compact complex manifolds in general; foliations on Inoue surfaces are perhaps the easiest counterexamples, see \cite[Remark 2.1]{MR2233706}. Also, the result does not hold for transversely affine germs of codimension one foliations, see \cite[sec. IV]{MR2008442}.

There were previous attempts to arrive at a structure theorem for singular transversely affine foliations of codimension one on projective spaces, see   \cite{MR1432053,MR1390971,MR1836428} to have a sample of such attempts. All these works approach the problem through the study of the foliation on a neighborhood of the singular divisor of the transverse affine structure based on an analysis of the (generalized) holonomy of this divisor, see also \cite{MR1711295}. They  use extension results to globalize the semi-local conclusions. The nature of this method leads one to impose restrictions on the type of singularities of the foliation. In  contrast, our approach is based on the study of the monodromy representation of the singular transversely affine foliation, and relies on recent results \cite{Bartolo:arXiv1005.4761, Budur:arXiv1211.3766} on the structure of representations of the fundamental groups of quasi-projective manifolds in the affine group $\Aff$. We also make use of some classical results  on the periods of families of closed rational $1$-forms  \cite{MR0417174} combined with basic properties of Picard-Fuchs equations; as well as results on the local/semi-local structure of singular transversely affine foliations.


As a rather concrete application,  we provide a classification of Liouvillian integrable $1$-forms on $\mathbb C^n$ which do not admit invariant algebraic hypersurfaces.

\begin{COR}\label{COR:B}
Let $\omega$ be a polynomial differential $1$-form on $\mathbb C^n$. If $\omega$ is Liouvillian integrable and
has no invariant algebraic hypersurface then there exists a polynomial map $P : \mathbb C^n \to \mathbb C^2$ and
polynomials $a,b \in \mathbb C[ x ]$ such that
\[
\omega = P^*( dy + ( a(x) + b(x) y ) dx ) \, .
\]
\end{COR}

The existence of such Liouvillian integrable $1$-forms have been recently recognized by \cite{MR2853194} as a new phenomenon,
but as stated above they are  nothing but  disguised classical Riccati equations.


Our Theorem \ref{THM:A} can be rephrased as a structure theorem for reducible  flat meromorphic $\mathfrak{sl}(2)$-connections  over projective manifolds.
By a  $\mathfrak{sl}(2)$-connection we  mean a connection  with  zero trace   on a rank two vector bundle with trivial determinant.

\begin{THM}\label{THM:C}
Let $X$ be a projective manifold with $h^1(X,\C)=0$.
Let $\nabla$ be a reducible flat meromorphic $\mathfrak{sl}(2)$-connection on a vector bundle $V$ over $X$.
 There exists a generically finite Galois morphism $p:Y\to X$ such that at least one of the following assertions holds true.
\begin{enumerate}
\item \label{COR:C1} The connection matrix of $p^*\mathcal \nabla$ in a suitable basis of rational sections of $p^*V$ is
\[
\left[ \begin{matrix} 0&\omega\\0&0\end{matrix}\right] \quad \text{or} \quad
\left[\begin{matrix} \eta/2&0\\0&-\eta/2\end{matrix}\right] \, .
\]
In particular the monodromy of $\nabla$ is virtually abelian.
\item \label{COR:C2} There exists  a curve $C$, a meromorphic flat connection $\nabla_0$ on a rank two bundle over $C$ and a rational map $\pi : Y \dasharrow C$ such that $p^*\nabla$ is birationally gauge equivalent to $\pi^*\nabla_0$.
Moreover, in this case the degree of $p$ is at most two.
\end{enumerate}
\end{THM}

\bigskip

\subsection*{Acknowledgements} We would like  to thank  Frank Loray. He contributed to this paper through numerous discussions and provided key ideas to the study of transversely affine foliations with monodromy in $(\mathbb C^*,\cdot) \subset \Aff$ which lead us to the  proof of Theorem \ref{T:mult}; he also suggested to rephrase Theorem \ref{THM:A} in terms of $\mathfrak{sl}(2)$-connections as is done in  Theorem \ref{THM:C}. We would  like to thank Hossein Movasati for explaining to us the basic properties of  Picard-Fuchs equations.

\tableofcontents

\section{Transversely affine foliations}

\subsection{Definition}\label{S:definition}
Let $\mathcal F$ be a codimension one holomorphic foliation on a complex manifold $X$ with normal bundle
$\calN{\F}$, i.e. $\mathcal F$ is defined by a holomorphic  section $\omega$ of $\calN{\F}\otimes \Omega^1_X$ with
zero locus of codimension $\geq 2$ and satisfying $\omega \wedge d \omega =0$. A   {\bf singular transverse affine structure} for $\mathcal F$ is a meromorphic flat connection
\[
\nabla : \calN\F \longrightarrow \calN\F \otimes \Omega^1_X(*D), \mbox{ satisfying } \nabla(\omega)=0; \,
\]
where $D$ is a reduced divisor on $X$ and $\Omega^1_X(*D)$ is the sheaf of meromorphic $1$-forms on $X$ with poles (of arbitrary order)
along $D$.
We will always take $D$ minimal, in the sense that the connection form of  $\nabla$ is not holomorphic in any point of $D$. The divisor $D$ is the {\bf singular divisor} of
the transverse affine structure.

A codimension one foliation $\mathcal F$ is a {\bf singular transversely affine foliation} if it admits a singular transverse affine structure. Aiming at simplicity, from now on we will
omit the adjective singular when talking about singular transverse  affine structures and singular transversely affine foliations.

As will be seen in Example  \ref{E:nonunique},  the same holomorphic foliation can admit more
than one transverse affine structure. When we want to keep track of the transverse affine
structure, we write ($\mathcal F, \nabla)$ instead of $\mathcal F$.

\subsection{Interpretation in terms of rational $1$-forms}\label{S:interpretation}
When $X$ is an algebraic manifold, the transverse affine structure can be defined by rational $1$-forms.
If $\omega_0$ is a rational $1$-form  defining $\mathcal F$  then the existence of a meromorphic flat connection on $\calN\F$ satisfying $\nabla(\omega)=0$
is equivalent to the existence of a rational $1$-form $\eta_0$ such that
\[
d \omega_0 =  \omega_0 \wedge \eta_0  \quad \text{ and } \quad d \eta_0 = 0.
\]
Indeed, if $U$ is an arbitrary open subset of a complex manifold $X$  where $\calN\F$ is trivial then a  flat meromorphic connection on a trivialization of  $\calN\F$ in $U$
can be expressed as
\[
\nabla_{|U} ( f ) = df +  f \otimes \eta_0 \, ,
\]
where $\eta_0$ is a closed meromorphic $1$-form which belongs to $H^0(U,\Omega^1_X(*D)_{|U})$. If $\omega_0$ represents $\omega$ on $U$
then $\nabla_{|U}(\omega_0) = d \omega _0 + \eta_0 \wedge \omega_0$ and $\nabla(\omega)=0$ is  equivalent to $d \omega_0 =  \omega_0 \wedge \eta_0$. If $X$ is
algebraic we can trivialize $N\mathcal F$ in the Zariski topology and get the sought pair of rational $1$-forms.

Most of time we will work with $U$ an affine open subset of a projective manifold $X$. At some points we will need to work with open subsets
in the analytic topology, as we are going to make use of results on the normal forms of singularities of codimension one foliations.

Notice that a change of trivialization does change $\eta_0$, and also changes $\omega_0$. If the pair $(\omega_0,\eta_0)$ represents $(\omega, \nabla)$ in a given trivialization
over $U$ then in another trivialization over $U$  the representatives will be of the form $(g\omega_0, \eta_0 - d \log g)$ for a suitable nowhere vanishing function $g \in \mathcal O_X(U)^*$.

The equality $d \omega_0 =  \omega_0 \wedge \eta_0 $ implies that the (multi-valued) $1$-form $ \exp(\int \eta_0) \omega_0$ is closed. Its primitives are first integrals for the
foliation $\mathcal F$. These first integrals belong to a Liouvillian extension of the field of rational functions on $X$, and conversely
the existence of a non-constant Liouvillian first integral for  $\mathcal F$  implies that $\mathcal F$ is transversely affine, see \cite{MR1062869}.

Even if $\omega_0$ and $\eta_0$ may have poles in the complement of $D$, the multi-valued function $\int \exp(\int \eta_0) \omega_0$ coincides with the developing map of $\F_{\vert X-(D\cup \sing\F)}$ and extends holomorphically to the universal covering of  $X-D$. Indeed, at a point $p$ in the polar set of $\eta_0$ or of $\omega_0$ which do not belong to $D$, we can choose another pair $(\omega_0',\eta_0')$ of rational $1$-forms, regular at $p$,
defining locally the foliation $\mathcal F$ and the connection $\nabla$.

For any given base point $q \in X -D$,  its  monodromy  is an anti-representation $\varrho$ of the fundamental group of the complement of
$D$ in $X$ to the affine group $\Aff = \mathbb C^* \ltimes \mathbb C$. The linear part of $\varrho$ will
be denoted by $\rho$. It coincides with the monodromy of $\nabla$.
\begin{center}
\begin{tikzpicture}
  \matrix (m) [matrix of math nodes,row sep=3em,column sep=4em,minimum width=2em]
  {
    \pi_1(X-D) & \Aff \\
      \, & \mathbb C^* \\};
  \path[-stealth]
    (m-1-1) edge node [above] {$\varrho$} (m-1-2)
            edge node [below] {$\rho$} (m-2-2)
    (m-1-2)  edge node [below] {} (m-2-2)  ;
\end{tikzpicture}
\end{center}
Here and throughout the paper we will deliberately omit the base point of the fundamental groups. This should not lead to any confusion.
Notice that  for any path $\gamma$ contained in the locus where both $\omega_0$ and $\eta_0$ are
regular we can write
\begin{align*}
\rho(\gamma) & = \left\{ z \mapsto z \cdot \exp\left( \int_\gamma \eta_0\right) \right\} \quad \text{ and }  \\
\varrho(\gamma) & = \left\{ z \mapsto z \cdot \exp\left( \int_\gamma \eta_0\right) +   \int_{\gamma} \exp \left( \int  \eta_0\right) \omega_0 \right\} \, .
\end{align*}

\subsection{Singular divisor and residues}\label{S:Singular}
Recall from the previous section that the singular divisor of  a transverse affine structure $\nabla$ is nothing but
 the reduced divisor of poles of $\nabla$.

\begin{prop}
The irreducible components of the singular divisor $D$ of a transverse affine structure $\nabla$
for a codimension one foliation $\mathcal F$ are invariant  by $\mathcal F$.
\end{prop}
\begin{proof}
Let $f$ be a local equation for an irreducible component $C$  of $D$ and $(\omega_0,\eta_0)$ be a local pair describing $(\F,\nabla)$
at a sufficiently small neighborhood of general point of $C$. The equation $d\eta_0=0$ imposes $\eta_0=\alpha+h(f)\frac{df}{f^k}$ with $\alpha$ a germ of holomorphic $1$-form
and $h$ a germ of  holomorphic function on $(\mathbb C,0)$ not vanishing at zero.
The  equation $d\omega_0=\omega_0\wedge \eta_0$ implies that $f$ divides $\omega_0 \wedge df$, i.e.  $C=\{f=0\}$ is $\F$-invariant.
\end{proof}

To each irreducible component $C$ of $D$ we can attach a complex number $\Res_C(\nabla)$,
defined as the residue
of any local meromorphic $1$-form $\eta_0$ defining $\nabla$ at a general point $p$ of $C$, i.e.
\[
\Res_C(\nabla) = \frac{1}{2i\pi}\int_\gamma \eta_0 \,
\]
for $\gamma$ equal to the boundary of a  disc intersecting $(\nabla)_{\infty}$ transversely  at $p$, and only at $p$.
The flatness of $\nabla$ (i.e. closedness of $\eta_0$) implies that this  complex number is independent of the choices of
$\eta_0$, $p$, and  $\gamma$.

\begin{prop}\label{P:log}
Let $X$ be a projective manifold. If $\nabla$ is any flat meromorphic connection on a line-bundle  $\mathcal L$ then  the class of
$-\sum \Res_{C}(\nabla) [C]$ in $H^2(X,\mathbb C)$,  with  the summation ranging over the irreducible components of the singular divisor $D$, represents the Chern
class of $\mathcal L$. Reciprocally, given a $\mathbb C$-divisor $R=\sum \lambda_C C$ with the same class in $H^2(X, \mathbb C)$ as a line bundle $\mathcal L$, there exists a flat meromorphic connection $\nabla_{\mathcal L}$ on
$\mathcal L$ with logarithmic poles and $\Res(\nabla_{\mathcal L})= - R$.
\end{prop}
\begin{proof}
When $X$ is a curve it is well-known that the Chern class of a line-bundle $\mathcal L$ with a meromorphic connection
$\nabla$ can be recovered from $-\sum \Res_{C}(\nabla) [C]$ in $H^2(X,\mathbb C)$, see for instance \cite[Chapter IV, Exercise 1.10]{Sabbah}. The general case can be proved by restriction of $\nabla$ to  general curves in $X$.

To realize a $\mathbb C$-divisor $R=\sum \lambda_C C$ as the residue divisor of a logarithmic connection with singular divisor $D= \sum  C$, we can replace
the pair $(X,D)$ by a log resolution since it suffices to construct the sought connection on $\mathcal L$ in the complement of a codimension two analytic subset and then extend it using Hartog's Theorem.

Suppose without loss of generality that  $D$ is a simple normal crossing divisor. We have the exact sequence
\[
0 \to \Omega^1_X \longrightarrow \Omega^1_X(\log D) \longrightarrow \oplus \mathcal O_C \to 0
\]
with the first arrow given by the inclusion and the second arrow given by the  residue map, cf. \cite[Chapter 6]{MR2114696}.
The boundary map $\oplus H^0(C, \mathcal O_C) \to H^1(X,\Omega^1_X)$ sends  $(\lambda_C)$ to  $\sum \lambda_C [ C]$, where
$[C]$ is the Chern class of $\mathcal O_X(C)$ in $H^1(X,\Omega^1_X) \subset H^2(X,\mathbb C)$. The inclusion of $H^1(X,\Omega^1_X)$
in $H^2(X,\mathbb C)$ is given by Hodge decomposition.

Suppose   $c_1(\mathcal L)=-c_1(R)=- \sum \lambda_C [C]$ in $H^1(X,\Omega^1_X)$. Let $s$ be a general rational section of $\mathcal L$. Here by general we
mean that  $D' = (s)_0 + (s)_{\infty} + D$ is a simple normal crossing divisor. Then there exists a logarithmic $1$-form $\eta$ with poles in $D'$ and residue divisor
equal to $R_0=R + (s)_{0} - (s)_{\infty}$. Since logarithmic $1$-forms on compact K\"ahler manifolds are closed \cite[Chapter 6]{MR2114696},
the connection $\nabla_0 = d + \eta$ is a flat connection on the trivial line-bundle with residue divisor equal to $R_0$.
Let $\nabla_1$ be the rational connection on $\mathcal L$ satisfying
$\nabla_1(s)=0$. If $s$ is given in a trivialization of $\mathcal L$ by a rational function $f$ then, locally, $\nabla_1 = d - \frac{df}{f}$. Therefore $\nabla_1$ is a flat logarithmic
connection with  residue divisor  $(s)_{\infty} - (s)_0$. The tensor product $\nabla_0\otimes \nabla_1$, obtained by summing up the local connection forms,
is a flat logarithmic  connection in $\mathcal L$  with residue divisor equal to $R_0  + (s)_{\infty} - (s)_0 = R$.
\end{proof}

\begin{remark}\label{R:trivial}
Two flat meromorphic connections $\nabla_1$ and $\nabla_2$  on the same line-bundle $\mathcal L$ differ
by a closed rational $1$-form, i.e. $\nabla_1 - \nabla_2 = \beta$ for $\beta$ a closed rational $1$-form. If the residues of $\nabla_1$ and $\nabla_2$ coincide then $\beta$ has no residues; in particular, when $h^1(X,\mathbb C)=0$, the rational  $1$-form $\beta$  is the differential of a rational function.
\end{remark}

\subsection{Examples and first properties} We collect below the standard examples  of transversely affine codimension one foliations and some basic properties concerning
the (non) uniqueness of transverse affine structure for a given foliation.

\begin{example}[Foliations with rational first integral]
If $F: X\dasharrow C$ is a dominant rational map to a curve, then $\omega_0=dF$ is a rational form which defines a transversely affine codimension one foliation.
It has many different transverse affine structures, see example \ref{E:nonunique} below.
\end{example}

\begin{example}[Foliations defined by closed $1$-forms]\label{E:nonunique}
If $\mathcal F$ is a codimension one foliation on a projective manifold defined by a closed rational $1$-form $\omega_0$,
then $\mathcal F$ admits a family of pairwise distinct  transverse affine structures parametrized by $\alpha \in \mathbb C$.  Indeed, for any constant $\alpha \in \mathbb C$ we have that $\eta_0 = \alpha \omega_0$
is closed and satisfies $d\omega_0 = \omega_0 \wedge \eta_0$. If $\alpha\neq 0$, since the monodromy is obtained by the analytic continuation of $F = \int  \exp(\int \alpha \omega_0) \omega_0 = \int \exp(\alpha G) dG= \exp(\alpha G)/\alpha$, where $G = \int \omega_0$, it must be of the form $\gamma \mapsto (\exp  \int_{\gamma} \omega_0)^{\alpha}\in \C^*$. If $\alpha=0$,  $\int  \exp(\int \alpha \omega_0) \omega_0=\int\omega_0$, and the monodromy is $\gamma \mapsto \int_{\gamma} \omega_0\in (\C,+)$.
\end{example}

\begin{prop}\label{uniqueness}
Let $\mathcal F$ be a codimension one foliation on a projective manifold $X$. Suppose $\F$ admits two distinct transverse affine structures.
Then $\F$ is defined by a closed rational $1$-form.   Moreover, if $\mathcal F$ does not admit a non-constant rational first integral then every transverse affine structure for $\mathcal F$ belongs to the one-parameter family presented in  Example \ref{E:nonunique}.
\end{prop}
\begin{proof}
If $\F$ admits two distinct transverse affine structures,
then for any  rational $1$-form $\omega_0$ defining $\mathcal F$, there exist two distinct closed rational $1$-forms $\eta_1$ and $\eta_2$ such that
\[
   d \omega_0 = \omega_0 \wedge \eta_i , \, i=1,2 \, .
\]
Therefore $\omega_0 \wedge (\eta_1 - \eta_2)=0$ and consequently $\eta_1 - \eta_2$ is a closed rational $1$-form defining $\mathcal F$.

If $\mathcal F$ is defined by a closed rational $1$-form $\omega_0$ and $(\omega_0,\eta_0)$ represents $(\F,\nabla)$, then $\eta_0$  must satisfy $\omega_0 \wedge \eta_0=0$. Therefore $\eta_0 = h \omega_0$ for a suitable rational function. Differentiation shows that $h$ must be constant along
the leaves of $\mathcal F$, i.e. $h$ is a rational first integral for $\mathcal F$. If $\mathcal F$ does not admit a non-constant rational first integral then
we are in the situation described in Example \ref{E:nonunique}. \end{proof}

If $\mathcal F$ is defined by a closed rational $1$-form $\omega_0$, the case $\alpha=0$ in Example $\ref{E:nonunique}$ says there exists an transversely affine  structure for $\mathcal F$
which has at worst logarithmic poles at the zeros and poles of $\omega_0$ and has additive monodromy group. The converse of this statement also holds true.

\begin{prop}
Let $(\mathcal F,\nabla)$  be a transversely affine codimension one foliation on a projective manifold $X$. If $\nabla$ has at worst logarithmic singularities
and the monodromy group of $(\mathcal F,\nabla)$ is contained in $(\mathbb C,+)$, then $\mathcal F$ is defined by a closed rational $1$-form.
\end{prop}
\begin{proof}
We have a locally well-defined closed meromorphic 1-form $\exp(\int\eta_0)\omega_0$ defining $\F$ in $X-D$, the monodromy hypothesis says it is well-defined in $X-D$.
We only have to check $\exp(\int\eta_0)\omega_0$ extends meromorphically through $D$. In a neighborhood $U$ of any smooth point of $D$, take a local pair $(\omega,\eta)$ representing $(\F,\nabla)$. Notice that  $\eta=\alpha+ \lambda \frac{df}{f}$ for a local equation $f$  of $D$, $\alpha$ closed  holomorphic $1$-form  in $U$ and $\lambda\in \Z$. There exists a meromorphic function $g$ on $U$ such that $(\omega_0,\eta_0)=(g\omega,\eta -\frac{dg}{g})$. Thus $\exp(\int\eta_0) \omega_0=g^{-1}\exp(\int \eta) g\omega=\exp (\int \alpha)f^\lambda \omega$ is meromorphic in $U$. Since $U$ is arbitrary, it follows that
$\exp(\int\eta_0)\omega_0$ is a well-defined meromorphic $1$-form on $X$.
\end{proof}

The hypothesis on the nature of the singularities of $\nabla$ is important in the proposition above.
There exist transversely affine codimension one foliations $\mathcal F$ on projective manifolds which have trivial  monodromy group but are not given by a closed rational $1$-form.

\begin{example}\label{llibre}
A  simple example on $\mathbb P^1 \times \mathbb P^1$ is given in affine
coordinates by
\[
\omega_0 = x^3dy   + 1/2(x+y)dx \quad \text{ and } \quad \eta_0 = - \frac{dx}{x} + \frac{dx}{x^3} \, .
\]
The only invariant curves are $\{x=0\}$ and the $\{ y=\infty \}$.   Proposition $\ref{uniqueness}$ implies that this  foliation cannot be given by a closed rational $1$-form.
It is birationally equivalent to the  one appearing in \cite[Theorem 3]{MR2853194}.
\end{example}

If $\mathcal F$ is a codimension one foliation defined by a closed rational $1$-form on a projective manifold $X$ and $\mathcal F$ is invariant by a finite group $G\subset \Aut(X)$,
then the quotient of $\mathcal F$ by $G$, seen in any resolution of $X/G$, is also a transversely affine foliation. Indeed, transversely affine
structures behave rather well under rational maps between foliations.

\begin{prop}
Let $X$ and $Y$ be projective manifolds, $f:X \dashrightarrow Y$ a dominant rational map, and $\mathcal F$ a codimension one foliation on $Y$.
The foliation $f^* \mathcal F$ has a transverse affine structure if and only if so does $\mathcal F$. If this occurs, the pull-back  of any
transversely affine structure  $(\F,\nabla)$ for $\mathcal F$ has as
monodromy group a finite index subgroup of the monodromy group of $(\F,\nabla)$.
\end{prop}
\begin{proof}
The result, phrased in terms of extensions of differential fields, is already implicit in \cite{MR1062869}. Except for the  finiteness of the index, a geometric
proof can be found in  \cite[Theorem 2.21]{MR2324555}.
We can write $f=\pi\circ h$ with $\pi: Z \dashrightarrow Y$ a generically finite rational
 map from a projective manifold $Z$ to $Y$, and $h:X \dashrightarrow Z$ a rational map with irreducible generic fiber.
The monodromy of the pull-back of the affine structure $(\mathcal F,\nabla)$ factors through $h$, and $h$ induces a surjective map of   fundamental groups, after restriction to any nonempty Zariski open set.
On the other hand, if we restrict $\pi$ over a sufficiently small nonempty Zariski open set, it induces a monomorphism with finite index image between fundamental groups.
\end{proof}

\begin{example}[Quotients]
Let $\mathcal F$ be a codimension one foliation on a projective manifold $X$ defined by a closed rational $1$-form $\omega_0$
and which does not admit a rational first integral.  If $\varphi \in \Aut(X)$ is an automorphism of finite order of $\mathcal F$ then
$\varphi^*(\omega_0)$  is also a  closed $1$-form  defining $\mathcal F$ and since $\mathcal F$ does not admit a rational first integral,
we must have $\varphi^* \omega_0 = \xi \omega_0$ for some root of unity $\xi$. The  quotient of $\mathcal F$ by $\varphi$ is a transversely affine foliation with
monodromy group equal to an extension of
the subgroup  of $\mathbb C^*$ generated by $\xi$
by a
subgroup of $(\mathbb C,+)$, determined by the integrals of $\omega_0$
along paths joining points in the same orbit of $\varphi$.
\end{example}

\begin{example}[Riccati foliations]
Let $X$ be a  projective manifold and $\pi:X \to Y$ a fibration with generic fiber isomorphic to $\mathbb P^1$. If $\mathcal F$ is a codimension one foliation on $X$
which has no tangencies with the general fiber of $\pi$ then we say that $\mathcal F$ is a Riccati foliation. An arbitrary  Riccati foliation
does not admit a transverse affine structure. Indeed, it follows from a classical result of Liouville   that a Riccati foliation admits a transverse
affine structure if and only if there exists a hypersurface $H\subset X$, invariant by $\mathcal F$ and which dominates $Y$, i.e. with $\pi(H) = Y$.
For example, if $H$ intersects the general fiber at only one point then there exists a birational transformation $\varphi : Y \times \mathbb P^1 \dashrightarrow X$
such that the strict transform of $H$ is the section at infinity of $Y \times \mathbb P^1 \to Y$. On $Y\times \mathbb P^1$ the foliation $\varphi^* \mathcal F$
is defined by a rational $1$-form $\omega_0=dy+ \alpha + y \beta$, with $\alpha, \beta$ pull-backs of rational $1$-forms on $Y$. Since $\omega_0$ is integrable it follows that $d\alpha=\alpha\wedge\beta$ and $\beta$ is closed.
If we take  $\eta_0=\beta$ then $d\omega_0=\omega_0\wedge \eta_0$ which shows that $\varphi^* \mathcal F$ is a transversely affine foliation.

It follows from the Riemann-Hilbert correspondence that there are no restrictions on the monodromy group of these Riccati foliations, see \cite{MR2337401}. In particular any finitely generated subgroup of $\Aff$ appears as
the monodromy group of a transversely affine Riccati foliation over $Y=\Pu$.
\end{example}

Notice that there exist transversely affine Riccati foliations with trivial monodromy but not given by a closed rational $1$-form, e.g. Example $\ref{llibre}$.
Similarly, there are Riccati foliations with trivial monodromy   which are not transversely affine foliations.

\subsection{Holonomy}
For a transversely affine codimension one foliation $(\F,\nabla)$ with singular divisor $D$, we define the {\bf singular leaves} of $\mathcal F$ as the leaves contained in $D$.
 Every other leaf of $\mathcal F$ will be called  a {\bf non singular leaf}.

\begin{prop}
The holonomy of any   non singular leaf $L$ of $\mathcal F$ is linearizable.
\end{prop}
\begin{proof}
Let $U=X-(D\cup\sing\F)$. Notice that the holonomy of any leaf of $\F_{\vert U}$ is in $\mathrm{Aff}(\C)$ and fixes a point in $\C$. Therefore
it is linearizable.
Since any element of $\pi_1(L)$ can be  represented by a loop in $L\cap U$, it follows that the same holds true for the  holonomy of $L$ .
\end{proof}

The determination of the holonomy of the singular leaves of $\mathcal F$ is more subtle and we will not treat the general case. 
For our purposes the statement below suffices.

\begin{prop}\label{P:holonomy}
Let $L$ be a smooth and irreducible component of $D$. Then the holonomy of $L$ is solvable. Moreover, 
if the singularity of $\nabla$ along $L$ is not logarithmic then the holonomy group of $L$ has a finite index subgroup  tangent to the identity,  and is thus virtually abelian.
\end{prop}
\begin{proof}
The fact that the  holonomy group  $L$ is a  solvable subgroup of $\Diff(\mathbb C,0)$ is well-known, see for instance \cite{MR1711295} or \cite{MR1869061}.
To prove the statement about non-logarithmic singularities of $\nabla$ we adapt the arguments of  \cite[pages 3076-3077]{MR1390971}.

Let $q \in L$ be a point where the foliation is smooth. In suitable local analytic coordinates at a neighborhood of $q$ the foliation $\mathcal F$ is defined by a $1$-form
$\omega=dy$ and the connection form is $\eta = \lambda \frac{dy}{y} + d(1/a(y))$ with $a(0)=0$.
Let $\Sigma$ be a transversal to $\mathcal F$ at $q$ with coordinate $y$.
The restriction of the (multi-valued) first integral $\int \exp(\int\eta) \omega$ to a sector with vertex at $q$ on $\Sigma$
is (one of the determinations of) $f(y)= \int_\star^y s^\lambda \exp(1/a(s)) ds$. If $h : (\Sigma,q) \to (\Sigma,q)$ is a holonomy map then
\[
f(y) = \alpha f(h(y)) + \beta
\]
for suitable $\alpha \in \mathbb C^*$ and $\beta \in \mathbb C$. After differentiating we can express $\alpha$ as
\[
\alpha = \frac{f'(y)}{f'(h(y))h'(y)} =
\frac{1}{h'(y)} \cdot \left(\frac{y}{h(y)}\right)^{\lambda} \cdot \exp\left( \frac{1}{a(y)} -\frac{1}{a(h(y))} \right)  \, .
\]
Since $\alpha$ is constant, the parameter in the exponential must be holomorphic as all the other factors in the product
have moderate growth. This implies $h'(0)^k= 1$, where $k$ is the vanishing order of $a(y)$ at $y=0$.
Thus the linear part of the members of $H$ are all roots of unity of order $\leq k$ and form a finite subgroup of $\mathbb C^*$. Therefore $H$ admits a finite index subgroup consisting of
germs of diffeomorphisms  tangent to the identity.  But solvable subgroups of $\Diff(\mathbb C,0)$ with trivial linear part are abelian, \cite[pages 3-4]{MR1692471}, and we conclude that
the holonomy of $L$ is virtually abelian.
\end{proof}

\subsection{Transversely affine foliations as transversely projective foliations}
A codimension one foliation $\mathcal F$ on a projective manifold $X$ is a {\bf
singular transversely projective foliation}  if there exists
\begin{enumerate}
\item  $\pi: P \to X$ a $\mathbb P^1$-bundle over $X$ locally trivial in the Euclidean topology;
\item $\mathcal H$ a codimension one singular holomorphic foliation
of $P$ transverse to
the generic fiber of  $\pi$;
\item $\sigma: X \dashrightarrow P$ a rational
section generically transverse to
$\mathcal H$;
\end{enumerate}
such that $\mathcal F = \sigma^* \mathcal H$.
The triple $\mathcal P= (P,\mathcal H,\sigma)$ is, by definition, a {\bf transverse projective structure} for $\mathcal F$.
This definition of singular transversely projective foliation is essentially equivalent to the one given in \cite{MR1432053},
for a comparison between the two definitions and thorough discussion see \cite{MR2337401}. As in the case of singular transverse affine structures/foliations we will deliberately
omit the adjective singular, and refer to this class of foliations from now on as transversely projective foliations.

Any two such triples $\mathcal P=(P,\mathcal H,\sigma)$ and $\mathcal P'=(P',\mathcal H',\sigma')$ are said {\bf birationally equivalent}
when they are conjugate by a birational bundle transformation
$\phi: P \dashrightarrow P'$ satisfying
$\phi^* \mathcal H'= \mathcal H$, and $\phi\circ \sigma= \sigma'$.

The {\bf polar divisor} of the transverse structure, denoted by $(\mathcal P)_{\infty}$, is the
divisor on $X$ defined by the direct image under $\pi$ of the
tangency divisor between $\mathcal H$ and the one-dimensional foliation
induced by the fibers of $\pi$.

Let $\vert (\mathcal P)_{\infty} \vert$ be the support of the polar divisor.
The {\bf monodromy representation} of  a projective structure $\mathcal{P}=(P\to X,\mathcal{H}, \sigma)$ is the (anti-)representation of
$\pi_1(X \setminus \vert (\mathcal P)_{\infty} \vert)$ into
$\mathrm{PSL}(2,\mathbb C)$ obtained by lifting paths on $X
\setminus \vert (\mathcal P)_{\infty} \vert$ to the leaves of
$\mathcal H$. Notice that the monodromy representation does not depend on $\sigma$.

Over a sufficiently small open subset $U$ of  $X$, the foliation $\mathcal H$ is the projectivization of a foliation on  a rank two vector bundle over
$U$ given by the flat sections of a meromorphic flat $\mathfrak{sl}(2)$-connection. We say that a transverse projective structure $(P,\mathcal H,\sigma)$ has {\bf regular singularities} when  the corresponding flat meromorphic connection is regular in the sense of \cite[Chapter II]{MR0417174}.

\begin{lemma}\label{L:obvious}
If we have two transverse projective structures $(P,\mathcal H,\sigma)$ and $(P',\mathcal H',\sigma')$ on $X$, both
having regular singularities then they have conjugate  monodromies if and only if there exists a birational bundle map $\phi: P \dashrightarrow P'$ such that $\phi^*\mathcal H' = \mathcal H$.
\end{lemma}
\begin{proof}
To compare the monodromies, we have to consider structures with the same polar locus $D$; to be in that situation we may take $D$ the union of both original polar loci. Once this is done,
suppose these monodromies are the same (or rather conjugated).

Over $X\setminus D$, there exists a $\Pu$-bundle isomorphism $\psi$ such that $\psi^*\mathcal H'_{\vert X \setminus D}=\mathcal H_{\vert X \setminus D}$, since both foliations are defined over $X\setminus D$ by the suspension
of the corresponding representations.
We want to show $\psi$ extends to $\phi$, a bimeromophic  $\Pu$-bundle map, defined on the whole of $X$. We first show that $\psi$ extends meromorphically in the neighborhood of any smooth point $q$ of $D$. Let $U$ be a sufficiently small  neighborhood of $q$ and $f$ a local equation in $U$ for the irreducible component $C$ of  $D$ through $q$.
We take $\mathfrak{sl}(2)$-connections $(E,\nabla)$, $(E',\nabla')$ which are local lifts for $(P,\mathcal H)$, $(P',\mathcal H')$ as in \cite[Section 2.1]{MR2337401}.

By regularity of the connections, and up to bimeromorphic transformations of $E_{\vert U}$, $E'_{\vert U}$ which are biholomorphic outside of  $E_{| C}$,  we can suppose that
both vector bundles are trivial and that the connections have at worst logarithmic poles. Moreover, by coincidence of monodromy, we can assume that the spectra $\{\theta/2,-\theta/2\}$, $\{\theta'/2,-\theta'/2\}$ of the residues matrices at $C$ are the same.  According to \cite[Remark 4.9]{MR2337401}, in suitable coordinates both $\mathcal H$ and $\mathcal H'$ can be defined by
\[
dz- \theta z\frac{df}{f} \quad  \text{ or }  \quad dz-(n z+ f^n)\frac{df}{f},
 \]
where $n=\theta \in \mathbb N$ in the second case.

Therefore  $\psi_{\vert U\setminus D}$ can be interpreted as  a symmetry $(x,\tilde{z}) \mapsto (x,A(f,z))$ of  $\mathcal H_{\vert U \setminus D}$
where $A(t,z)=(a(t)z + b(t))(c(t)z+d(t))^{-1}$ is a Moebius transformation with coefficients $a,b,c,d$
holomorphic in a punctured neighborhood of $0 \in \mathbb C$. In any case, plugging the coordinate
functions of $A$ in the equations for the symmetry of $\mathcal H_{\vert U \setminus D}$ shows that $A$ is meromomorphic at $f=0$.
Since $q \in D$  is an arbitrary smooth point, it follows that we can meromorphically extend $\psi$ to the complement of a codimension two subset of $P$.
 Levi extension theorem allow us to meromorphically extend $\psi$ to the whole $X$, and obtain the sought birational bundle map $\phi$.
 \end{proof}

\medskip

Every transversely affine codimension one foliation $(\mathcal F, \nabla)$ on a projective manifold $X$ carries a natural transverse projective structure
$\mathcal P_{\nabla}$.
It is given by $(P,\mathcal H,\sigma)$ as follows.
\begin{itemize}
\item $P=\mathbb P(E)$ is the family of lines in $E$  where $E=\calN\F \oplus \mathcal O_X$;
\item $\sigma: X \to P$ the section  corresponding to the inclusion $\mathcal O_X  \to \calN\F \oplus \mathcal O_X$;
\item $\mathcal H$ the foliation on $P$ defined by projectivization of a flat connection $\D$ on $E$; its leaves are the projections of horizontal sections of $\D$.
We now describe $\D$. We have a flat connection $\fnabla$ on $E$ given by   $ \fnabla=\nabla \oplus d$, we also have a map
$i:\calN\F\otimes \Omega^1_X \rightarrow \mathrm{End}(E)\otimes \Omega^1_X$ induced by the composition of natural maps $$\calN\F\simeq \Hom(\mathcal O_X, \calN\F)\hookrightarrow \mathrm{End}(E);$$
let $\omega\in \mathrm{H}^0(X,\calN\F\otimes\Omega_X^1)$ be a section defining $\F$, we define $\D$ to be the translated of $\fnabla$ by $i(\omega)$: $\D=\fnabla+i(\omega)$. It is easily checked that $\D$ is flat.
\end{itemize}
If we perform the construction of $\D$ starting with another section $\lambda \omega$, $\lambda \in \C^*$, then we obtain $\D'$, the transform of $\D$ by the automorphism $\lambda \oplus1$ of $E$. So that the isomorphism class of $\D$ is canonically defined by the  transverse affine structure of $\F$. 

If we use a  trivialization coordinate $z :\calN\F_{\vert U}\rightarrow \C$ on $\calN\F_{\vert U}$ and  $z\oplus id$ on $E_{\vert U}$,
we see that $\D=d+\Omega$ has connection matrix $\Omega=\left [ \begin{smallmatrix} \eta & \tilde{\omega} \\ 0&0 \end{smallmatrix} \right ]$, where
$\tilde{\omega}$ and $\eta$ represent respectively $\omega$ and  the connection matrix of $\nabla$ in the trivializations.

 In such a trivialization, $\mathcal H$ coincides  with the foliation defined by
the meromorphic $1$-form
\[
dz + \tilde{\omega}+ z \eta
\]
on the open subset $U\times \mathbb{P}(\calN\F_{\vert U}\oplus 1)$.
Over $U$, the section $\sigma$ is $z=0$.

\begin{prop}\label{P:pullRiccati}
Let $\mathcal F$ be a transversely projective codimension one foliation on a projective manifold $X$ with transverse projective structure $\mathcal P=(P,\mathcal H, \sigma)$. Suppose there exists a fibration $f: X \to C$ with connected fibers such that the (local) meromorphic flat connection defining $\mathcal H$ has regular singularities along the general  fiber of $f$
and that
the monodromy representation $\rho$ of $\mathcal P$ factors through $f$, i.e.  there exists a divisor $F$ supported on finitely many  fibers of $f$ and a representation $\rho_0$ from the fundamental group of $C_0= f(X -  |(\mathcal P)_{\infty} + F|)$ to $\mathrm{PSL}(2,\mathbb C)$ fitting in the diagram below.
\begin{center}
\begin{tikzpicture}
  \matrix (m) [matrix of math nodes,row sep=2em,column sep=4em,minimum width=2em]
  {
    \pi_1(X - |(\mathcal P)_{\infty} + F| ) & \PSL  \\
      \pi_1(C_0 ) & \\};
  \path[-stealth]
    (m-1-1) edge node [above] {$\rho$} (m-1-2)
            edge node [left] {$f_*$} (m-2-1)
    (m-2-1)  edge node [below] {$\rho_0$} (m-1-2)  ;
\end{tikzpicture}
\end{center}
Then there exists a $\mathbb P^1$-bundle $S$ over $C$;  a Riccati foliation
$\mathcal R$ on $S$; and a rational map $p: X \dashrightarrow S$ such that $p^* \mathcal R = \mathcal F$.
\end{prop}
\begin{proof}
Let $\pi$ denote the projection of $P$, and  $\G$ be the codimension two foliation of $P$ obtained as the intersection of $\mathcal H$ with the foliation determined by $f$.
The leaves of $\mathcal G$  are the leaves of the restrictions $\H_{\vert f^{-1}(y)}$ over the fibers of $f$.
For generic $y\in C$, $\H_{\vert f^{-1}(y)}$ has regular singularities and trivial monodromy, hence according to Lemma \ref{L:obvious} it is birationaly equivalent to  the trivial horizontal codimension one foliation on the trivial $\Pu$-bundle over $f^{-1}(y)$.
The main result of \cite{MR1017286} (see also \cite[Section 8]{MR1860669})  implies that $\G$ is defined by the levels of a rational dominant map with connected fibers $F : P \dasharrow S$, with $S$ a smooth projective surface.
By construction, $f\circ \pi$ factors through $F$ : $f\circ \pi=q\circ F$, for some rational map $q:S \dasharrow C$. Up to birational transformation of $S$, we can suppose $q$ is holomorphic.
It follows from \cite[Lemma 3.1]{MR2324555} that $\H$ projects to a foliation $\mathcal R$ on $S$ such that $F^* \mathcal R= \H$.

For general $x\in X$, the restriction of $F$ to $\pi^{-1}(x) \simeq \mathbb P^1$ is not constant, and it separates the points of $\pi^{-1}(x)$ as they correspond to different leaves of $\mathcal G$. Therefore
$F$ takes a general fiber of $\pi$ biholomorphically into a fiber of $q$. We conclude that $\mathcal R$ is a Riccati foliation on $S$, with adapted fibration $q$. Defining  $p:=F\circ \sigma$ yields the conclusion.
 \end{proof}

\begin{prop}\label{P:pushRiccati}
Let $X$ and $Y$ be projective manifolds, $f:X \dashrightarrow Y$ a dominant rational map, and $\mathcal F$ a codimension one foliation on $Y$.
If $f^* \mathcal F$  is a pull-back of a transversely affine Riccati equation on a surface  then either  $\mathcal F$ is  a pull-back of a transversely affine Riccati equation on a surface, or there exists a
generically finite morphism $g: Z \to Y$ such that $g^* \mathcal F$ is given by a closed rational $1$-form.
\end{prop}
\begin{proof}
It suffices to consider the case where $\dim X= \dim Y$ since we can  replace $X$ by a general submanifold  with the same dimension as $Y$.  Let $r:Z \dashrightarrow X$ be a dominant rational map between manifolds of the same dimension such that the composition $g=f\circ r$ defines a Galoisian field extension  $g^*:\mathbb C(Y) \to \mathbb C(Z)$, i.e., the group of birational transformations
$\varphi :Z \dashrightarrow Z$ which satisfy $g\circ \varphi = g$ acts transitively on the  general fiber of $g$. Notice that $g^*\mathcal F$ admits a transverse affine structure and  is the pull-back of a Riccati foliation on a surface.

If the transverse affine structure for $g^*\mathcal F$ is not unique then $g^*\mathcal F$ is defined by a closed rational $1$-form according to Proposition \ref{uniqueness}.

If the transverse affine structure for  $g^* \mathcal F$
is unique then it must be  invariant under birational maps $\varphi : Z \dashrightarrow Z$ such that $g\circ \varphi = g$. In other words,
if we consider the projective structure $(P,\mathcal H,\sigma)$ naturally associated to $g^* \mathcal F$ then every birational
deck transformation $\varphi$ of $g$ lifts to a birational map $\Phi :P \dashrightarrow P$ which preserves $\mathcal H$
and $\sigma$, i.e. $\Phi^* \mathcal H= \mathcal H$ and  $\Phi\circ \sigma=\sigma$. Since $g^* \mathcal F$ is a pull-back of a Riccati foliation $\mathcal H_0$ on a surface $S$, the same holds true for $\mathcal H$.
Notice that
the fibers of the pull-back map $(P,\mathcal H) \dashrightarrow (S, \mathcal H_0)$ define a codimension two foliation $\mathcal G$ by algebraic subvarieties tangent to $\mathcal H$.
To prove that $\mathcal F$ is also a pull-back of a Riccati  foliation on a surface it suffices to verify that $\mathcal G$ is invariant
by $\Phi$. Since $\mathcal H$ is invariant by $\Phi$, if  $\Phi^* \mathcal G \neq \mathcal G$ then $\Phi^* \mathcal G$ would be
another foliation by codimension two algebraic subvarieties tangent to $\mathcal H$. But this would imply that the leaves of $\mathcal H$
are algebraic and the same would hold true for $g^* \mathcal F$ and $\mathcal F$. This contradiction shows that $\Phi^* \mathcal G= \mathcal G$, and
therefore $\mathcal F$ is also a pull-back of a Riccati foliation on a surface under a rational map.
\end{proof}

\section{Cohomology jumping loci for local systems}

Let $X$ be a projective manifold and $U \subset X$ be the complement of a divisor $D$ of $X$.
In this section we are going to review results on the structure of representations
\[
\varrho : \pi_1(U) \longrightarrow \Aff
\]
which will be essential in what follows.

\subsection{Group cohomology}

Let  $\Gamma$ be a group, $V$ a finite dimension vector space, and $\rho: \Gamma \to \GL(V)$ a morphism of groups.  The homomorphism $\rho$ endows $V$ with the
structure of a $\Gamma$-module which we will denote by $V_{\rho}$. The first cohomology group of $\Gamma$ with values in $V_{\rho}$ can be defined as the quotient
of $1$-cocycles and $1$-coboundaries
\[
H^1(\Gamma, V_{\rho} ) = \frac { \left\{ \varphi : \Gamma \to V ; \varphi(\gamma_1 \cdot \gamma_2) = \varphi(\gamma_1) + \rho(\gamma_1) \varphi(\gamma_2) \right\} }
{\left\{ \varphi : \Gamma \to V ; \exists v \in V \text{ such that } \varphi(\gamma)= \rho(\gamma)v - v \, , \forall \gamma \in \Gamma\right\}  } \, .
\]

Let  $\varrho : \Gamma \to \Affn$ a representation in the affine group $\mathrm{Aff}(V) = \GL(V) \ltimes V$.
If $\gamma$ belongs to $\Gamma$ then we can write $\varrho(\gamma)(z) = \rho(\gamma)z + \tau(\gamma)$, where
$\rho : \Gamma \to \GL(V)$ is a homomorphism; and $\tau : \Gamma \to V$ is a $1$-cocycle with values in  $V_{\rho}$.
The class of $\tau$ in $H^1(\Gamma, V_{\rho})$ is trivial if and only if the action of $\varrho(\Gamma)$ fixes a point, so that
we can write $\varrho(\gamma)(v) = \rho(\gamma)\cdot(v - v_0) +v_0$. If $\rho$ is the trivial homomorphism then
$H^1(\Gamma, V_{\rho}) = \Hom(\Gamma, V)$.

\subsection{Cohomology jumping loci for quasi-projective manifolds}
Let  $U$ be a  quasi-projective manifold.  If $\Gamma=\pi_1(U)$ then $H^1(\Gamma, V_{\rho})=H^1(U,\mathbb C_{\rho})$, where $\mathbb C_{\rho}$ is the rank one local system on $U$ having $\rho$ as its monodromy.
The characteristic varieties of $U$ are defined as
\[
\Sigma^i_k(U) = \{ \rho \in  \Hom(\pi_1(U), \mathbb C^*) ; \dim H^i(U,\mathbb C_{\rho}) \ge k \} \, .
\]
When $U$ is compact they have been studied by Green-Lazarsfeld, Beauville, Catanese,  Simpson, Campana, Delzant
and others, and when $U$ is not proper they have been studied by Arapura, Dimca, Bartolo-Cogolludo-Matei,   Budur-Wang and others.
See  \cite{Bartolo:arXiv1005.4761}, \cite{Budur:arXiv1211.3766} and references therein.

Of particular interest for us, is the first characteristic variety which is described by the following theorem which combines
results by Arapura and Bartolo-Cogolludo-Matei, and is stated in \cite{Bartolo:arXiv1005.4761} in a slightly different form which
we present afterwards.

\begin{thm}\label{T:budur}
If $U$  is   a quasi-projective manifold then
$\Sigma^1_k(U)$ is  a finite union of torsion translates of subtori of $\Hom(\pi_1(U), \mathbb C^*)$.
Moreover,  each irreducible component of $\Sigma^1_k(U)$ of positive dimension is a  translate by a torsion element of a
subtorus of the form
\[
f^* \Hom(\pi_1(C), \mathbb C^*)
\]
where $C$ is a quasi-projective curve and $f : U \to C$ is a morphism.
\end{thm}

In particular, if $\rho: \pi_1(U) \to \mathbb C^*$ belongs to a positive dimensional component of $\Sigma^1_k(U)$
then there exists an \'{e}tale covering $p:V \to U$ and $\rho':\pi_1(C) \to \mathbb C^*$ such that the following
diagram commutes.
\begin{center}
\begin{tikzpicture}
  \matrix (m) [matrix of math nodes,row sep=2em,column sep=4em,minimum width=2em]
  {
    \pi_1(V) & \pi_1(C) \\
      \pi_1(U) & \mathbb C^* \\};
  \path[-stealth]
    (m-1-1) edge node [above] {$(f\circ p)_*$} (m-1-2)
            edge node [left] {$p_*$} (m-2-1)
    (m-1-2) edge node [right] {$\rho'$} (m-2-2)
    (m-2-1)  edge node [above] {$\rho$} (m-2-2)  ;
\end{tikzpicture}
\end{center}

The covering $p$ is determined by the torsion character used to translate the subtorus and
is related to the presence of multiple fibers of the fibration $f:U\to C$.
This factorization is more succinctly stated in the language of orbifolds: there exist $C$ an orbifold of dimension one, $f:U \to C$ a morphism of orbifolds, and  a representation $\rho': \pi_1^{orb}(C) \to \mathbb C^*$
such that
\[
\rho = \rho' \circ f_* \, .
\]
This is the statement of \cite[Theorem $1$]{Bartolo:arXiv1005.4761}.

\section{Factorization of representations}
The result below is an easy consequence of Theorem \ref{T:budur}, and is  well-known to the
specialists. Indeed, in \cite[Theorem 5.1]{Bartolo:arXiv1005.4761} it appears as an important intermediate step toward the proof of Theorem \ref{T:budur}. Nevertheless, we present a proof using Theorem \ref{T:budur} as a black-box, since the argument is short and clarifies what sort of  obstructions one may find when trying to factorize a representation in $\Aff$ through a fibration.

\begin{thm}\label{T:factoriza}
Let $U$ be a quasi-projective manifold  and $\varrho : \pi_1(U) \to \Aff$ be
a representation in the affine group. If the image of $\varrho$ is Zariski dense in $\Aff$
then there exists an orbifold $C$ of dimension one, a morphism of orbifolds $f: U \to C$, and
a representation $\tilde{\varrho} : \pi_1^{orb}(C) \to \Aff$ factoring $\varrho$ as in the diagram below.
\begin{center}
\begin{tikzpicture}
  \matrix (m) [matrix of math nodes,row sep=2em,column sep=4em,minimum width=2em]
  {
    \pi_1(U) & \Aff  \\
      \pi_1^{orb}(C ) & \\};
  \path[-stealth]
    (m-1-1) edge node [above] {$\varrho$} (m-1-2)
            edge node [left] {$f_*$} (m-2-1)
    (m-2-1)  edge node [below] {$\tilde{\varrho}$} (m-1-2)  ;
\end{tikzpicture}
\end{center}
\end{thm}
\begin{proof}
Let $\rho : \pi_1(U) \to \mathbb C^*$ be the linear part  of $\varrho$, i.e. $\rho$ is the composition of
$\varrho$ with the natural projection
$\Aff \to \mathbb C^*$. Since $\varrho$ has Zariski dense image, it must be non abelian  and therefore
$k = h^1(U, \mathbb C_{\rho})>0$.  Since $\rho$ is not torsion, Theorem~\ref{T:budur} implies
that the germ $\Sigma_{\rho}$ of  $\Sigma^1_k(U)$ at $\rho$ is smooth of positive dimension. Indeed, if there are two
distinct irreducible components through $\rho$ then $\rho$ factors through two distinct fibrations and the general fiber
of one of the fibrations dominates the basis of the other fibration. Since the representation is the identity over the general fiber
of both fibrations, \cite[Lemma 4.19]{Debarre} implies that the representation has finite image, i.e., $\rho$ is torsion. Therefore every
$\rho' \in \Sigma_{\rho}$ satisfies $h^1(U, \mathbb C_{\rho'}) =h^1(U, \mathbb C_{\rho})$.
Consequently
we have a morphism $\Psi : \Sigma_{\rho} \to \Hom(\pi_1(U),\Aff)$ such that
$\Psi(\rho) = \varrho$ and $\Psi(\rho')$ is a representation with  linear part $\rho'$,  see \cite[Lemma 2.1]{MR1240576}.

Let $f:U \to C$ be the morphism of orbifolds given by Theorem \ref{T:budur}.
Let $U_0 \subset U$ be a Zariski open subset such that the restriction of $f$ to $U_0$ is
a smooth fibration, locally trivial in the $C^{\infty}$ category, over $C_0= f(U_0)$. Let us
compare the long exact sequence for the homotopy groups of a fibration
with the factorization of the linear
part of $\varrho$.
\begin{center}
\begin{tikzpicture}
  \matrix (m) [matrix of math nodes,row sep=2em,column sep=4em,minimum width=2em]
  {
   0 & (\mathbb C,+) & \Aff & \mathbb C^* & 1  \\
   0 & \pi_1(F) & \pi_1 (U_0) & \pi_1(C_0 ) & 1\\};
  \path[-stealth]
    (m-1-1) edge (m-1-2)
    (m-1-2) edge (m-1-3)
    (m-1-3) edge (m-1-4)
    (m-1-4) edge (m-1-5)
    (m-2-1) edge (m-2-2)
    (m-2-2) edge (m-2-3)
    (m-2-3) edge node [above] {$f_*$} (m-2-4)
    (m-2-4) edge (m-2-5)
    (m-2-2) edge node [left] {$\varrho$}  (m-1-2)
    (m-2-3)     edge node [left] {$\varrho$} (m-1-3)
    (m-2-4)  edge node [left] {$\rho$} (m-1-4)  ;
\end{tikzpicture}
\end{center}
From this diagram we deduce that $\varrho(\pi_1(F))$ is a finitely
generated normal subgroup of $\varrho(\pi_1(U_0))= \varrho(\pi_1(U))$
with trivial linear part. If $\tau \in  \varrho(\pi_1(F)) \subset (\mathbb C,+)$
and $\lambda \in \rho(\pi_1(U_0))$, then by conjugation
\[
 \lambda^i \tau \in \varrho\left(\pi_1(F)\right)
\]
for every $i \in \mathbb Z$. Since $\varrho(\pi_1(F))$ is finitely generated,
if $\tau \neq 0$  then both $\lambda$ and $\lambda^{-1}$ are roots of polynomials with integer coefficients.
Therefore  either
$\varrho(\pi_1(F))= 0 $ or   $\rho(\pi_1(U_0))$ is contained
in the ring of algebraic integers of some number field $K \subset \mathbb C$.

Notice that  a general $\rho' \in \Sigma_{\rho}$
is not defined over a number field and therefore the representation $\Psi(\rho')$ factors.
Since the factorization is equivalent to the triviality
of $\Psi(\rho')$ over a general fiber of $f$ and this is a closed property, it follows that the representation $\varrho=\Psi(\rho)$
also factors as wanted.
\end{proof}

\section{Proof of Theorem \ref{THM:A} }

Let $(\mathcal F,\nabla)$ be  a transversely affine codimension one foliation on a projective manifold $X$ with singular divisor $D$ and complement $U=X-D$.
Let $\varrho: \pi_1( U ) \to  \Aff$ be its monodromy representation and
$\rho : \pi_1( U ) \to \C^*$ be its multiplicative part.

We will divide the proof of Theorem \ref{THM:A} according to the properties of  $\varrho$.

\subsection{Zariski dense monodromy} Under the assumption that $\varrho$ has Zariski dense image we are able
to prove Theorem \ref{THM:A} on arbitrary projective manifolds as already mentioned in the introduction.

\begin{thm}
Let $X$ be a projective manifold and let $(\mathcal F,\nabla)$ be  a transversely affine codimension one foliation on $X$.  If the monodromy of $(\mathcal F, \nabla)$
 is Zariski dense in $\Aff$ then  there exist a transversely affine Ricatti foliation $\mathcal R$ on a projective surface $S$ and
a rational map $p:X \dashrightarrow S$ such that $p^* \mathcal R = \mathcal F$.
\end{thm}
\begin{proof}
Suppose that $\varrho$ has Zariski dense image in $\Aff$. Theorem \ref{T:factoriza} implies that
the representation $\varrho$ factors through a morphism of orbifolds $f:U \to C_0$ where $C_0$ is a quasi-projective orbicurve.
Include in $D$ the fibers of $f$ over the  multiple fibers of $f$. Over the new $U=X-D$,
$f$ is just a regular morphism to a quasi-projective curve $C_0$, restriction of a rational map $f: X \dashrightarrow C$
between projective manifolds. Modulo  resolving  the indeterminacies of this map, we can assume that $f : X \to C$ is regular, and
its restriction to a Zariski open subset $U$ factors the monodromy of $(\mathcal F,\nabla)$ through a quasi-projective curve $C_U= f(U)$, i.e.
there exists $\varrho_f : \pi_1(C_U) \to \Aff$ such that $\varrho_f \circ f_* =  \varrho$.

In order to be able to apply Proposition \ref{P:pullRiccati} we have to exclude the existence of  irreducible components
$H$ of the singular set of $\nabla$ which are not logarithmic and have image under $f$ dominating $C$.
We argue by contradiction and assume the existence of an irreducible component $H$ in the polar set of $\nabla$ which is not logarithmic
and which dominates $C$.
Since the  monodromy of $(\mathcal F,\nabla)$ factors through $C_U$, it induces a  representation  $\varrho_H : \pi_1(U_H) \to \Aff$, where $U_H = f^{-1}(f(U)) \cap H$, and $\varrho_H = \varrho_f \circ (f_{|H})_*$. Moreover, according to Proposition \ref{P:holonomy},
there exists a finite index subgroup $G$ of $\pi_1(U_H)$ whose image under the holonomy representation is abelian.

 Since elements in $[G,G]$ have trivial holonomy, representatives of them lift to leaves of $\mathcal F$ nearby $H$, and consequently $\varrho_H([G,G])=\{\mathrm{id}\}$. To arrive at a contradiction with the density of the monodromy group notice that  $f_*\pi_1(U_H)$ has finite index in $\pi_1(C_U)$ according to \cite[Lemma 4.19]{Debarre}.  This proves that such an $H$ cannot exist, and the theorem follows from Proposition~\ref{P:pullRiccati}.
\end{proof}

\subsection{Virtually additive monodromy}

\begin{thm}\label{T:add}
Let $X$ be a projective manifold with $h^1(X,\mathbb C)=0$ and let $(\mathcal F,\nabla)$ be a transversely affine codimension one foliation on $X$.  If the monodromy of $(\mathcal F, \nabla)$  is contained in a finite extension of $(\mathbb C,+) \subset \Aff$,  then either there exists a generically finite Galois morphism $p:Y\to X$ such that $p^*\mathcal F$ is defined by a closed rational $1$-form or there exist a transversely affine Ricatti foliation $\mathcal R$ on a surface $S$ and a rational map $p:X \dashrightarrow S$ such that $p^* \mathcal R = \mathcal F$.
\end{thm}
\begin{proof}
If $\nabla$ is logarithmic, with connection form $\eta_0$ in a Zariski open set, then $ \exp(\int \eta_0) $ is a multi-valued algebraic function. Its branches determine a generically finite Galois morphism  $p:Y\to X$ such that $p^*\mathcal F$ is defined by a closed rational $1$-form.

Suppose  that $\nabla$ is not logarithmic.  Proposition \ref{P:log} implies the existence of a logarithmic connection $\nabla_{\log}$ on $N\mathcal F$ having the same residues as $\nabla$.
As explained in Remark \ref{R:trivial}, the difference $\nabla_{\log} - \nabla$ is a closed rational $1$-form $\beta$   without residues. Since $h^1(X,\mathbb C)=0$, the $1$-form $\beta$ is exact in the sense that there exists a rational function $g \in \mathbb C(X)$ such that $\beta = dg$. We can assume that $g$ defines a regular morphism   from $X$ to $\mathbb P^1$. Let $h: X \to \mathbb P^1$ be the Stein factorization of $g$. Of course the target is still $\mathbb P^1$ as we are assuming $h^1(X,\mathbb C)=0$.

We want to show that the monodromy representation factors through $h$. To that end suppose first that all the residues of $\nabla$ are integers. In this case we can choose a pair of rational $1$-forms  $\omega_0, \eta_0$ in a Zariski open subset $U$ such that  $\eta_0 = h^* \beta_0$ for some rational $1$-form $\beta_0$ on $\mathbb P^1$. The equation $d\omega_0 = \omega_0 \wedge h^* \beta_0$ implies that $\omega_0$ is closed when restricted to the general  fiber of $h$. To prove the factorization of the monodromy it suffices the restriction of $\omega_0$ to a general fiber of $h$ is not only closed, but exact.

Let $p$ be a non-logarithmic pole  of $\beta_0$ and let us suppose that $U$ intersects the fiber over $p$, otherwise we can start with a different $U$.  Replacing $U$ by a smaller open subset we can suppose that $h$ restricted to  $U^* = U \setminus h^{-1}(p)$ is a locally trivial $C^{\infty}$-fibration over $T^*= h(U^*)$. The $1$-form $\omega_0$ can be interpreted as  a family $\omega_{0,t}$  of closed rational $1$-forms on the quasi-projective manifolds $U_t = U \cap h^{-1}(t)$ parametrized by $t \in T^*$. We want to prove that for a general $t \in T^*$ and any $\delta_t \in H_1(U_t,\mathbb Z)$ the integral $\int_{\delta_t} \omega_{0,t}$ is equal to zero.

The multi-valued function
\begin{align*}
  F : T^*  &\longrightarrow \mathbb C  \\
     t &\longmapsto \int_{\delta_t} \omega_{0,t}
\end{align*}
obtained by continuous deformation of $\delta_t \in H_1(U_t,\mathbb Z)$ satisfies the so called  Picard-Fuchs equation, see \cite[Chapters $10$ and $12$]{MR2919697} specially \S$10.2.4$ and \S$12.2.1$.
Let us fix an arbitrary point $t_0 \in T^*$. For $t$ sufficiently close to $t_0$, let $\gamma_t$ be a real curve in $T^*$ joining $t_0$ to $t$,  let  $\Delta_t$ be
 a real two-dimensional surface on $U$ which projects to  $\gamma_t$, intersects fibers of $h$ over a  point $s$ in the path $\gamma_t$
 at the cycle $\delta_s$, and consequently has boundary equal to $\delta_t - \delta_{t_0}$.
Then
\begin{align*}
dF & = \frac{\partial}{\partial t}\left( \int_{\delta_t} \omega_{0,t} \right) dt   = \frac{\partial}{\partial t} \left( \int_{\Delta_t} d \omega_{0} +  \int_{\delta_{t_0}} \omega_{0,t_0} \right) dt \\
 &=  \frac{\partial}{\partial t} \left( \int_{\Delta_t} \omega_0 \wedge h^* \beta_0 \right) dt   \\
 & = \frac{\partial}{\partial t} \left( \int_{\gamma_t} \left( \int_{\delta_t} \omega_0 \right)  \cdot  \beta_0 \right) dt  = F \cdot \beta_ 0   \,
\end{align*}
where we have used Stokes Theorem in the first line, and Fubini Theorem to pass from the second to the third line.  Thus in our setting,
the Picard-Fuchs equation is nothing but
 \[
 \frac{dF}{F} = \beta_0 \, .
 \]
Therefore, if one of the periods of $\omega_{0,t}$ is  not zero then the function $F(t)$ does not have moderate growth when
we approach $p$. But this contradicts   \cite[Th\'{e}or\`{e}me 1.8, page 125]{MR0417174} (see also \cite[Theorem 12.3]{MR2919697}) which   roughly says that the periods of holomorphic families of rational $1$-forms are functions with moderate growth at infinity. We conclude that periods of $\omega_{0,t}$ are zero for any $t \in T^*$, i.e. $\omega_0$ is exact on the general fiber of $h$ and therefore the monodromy factors through $h$. We apply Proposition \ref{P:pullRiccati} to conclude that $\mathcal F$ is a pull-back of a Riccati foliation over a rational surface.

If the residues of $\nabla$ are not integers in general, they must be rational and we can apply the above arguments to the pull-back of $\mathcal F$ under a generically finite rational map $p:Y \to X$  to conclude that $p^* \mathcal F$ is the pull-back of a Riccati foliation on a surface. Proposition \ref{P:pushRiccati} allows us to conclude.
\end{proof}

\subsection{Multiplicative monodromy}

\begin{thm}\label{T:mult}
Let $X$ be a projective manifold with $h^1(X,\mathbb C)=0$ and let $(\mathcal F,\nabla)$ be a transversely affine codimension one foliation on $X$.  If the monodromy of $(\mathcal F, \nabla)$ is contained in $(\mathbb C^* ,\cdot) \subset \Aff$  then either there exists a generically finite Galois morphism $p:Y\to X$ such that $p^*\mathcal F$ is defined by a closed rational $1$-form or there exists a transversely affine Ricatti foliation $\mathcal R$ on a surface $S$ and a rational map $p:X \dashrightarrow S$ such that $p^* \mathcal R = \mathcal F$.
\end{thm}
\begin{proof}
Let $(P,\mathcal H, \sigma)$ be the transverse projective structure naturally associated to  $(\mathcal F,\nabla)$.
Since the monodromy is multiplicative, on the complement of the singular divisor of $\nabla$ we have two sections $\Sigma_1, \Sigma_2$ of $P$ invariant by $\mathcal H$. If $\nabla$ is logarithmic then these sections are indeed meromorphic over all $X$. If we choose a trivialization of $P$ where these sections are at zero and at infinity then $\mathcal H$ is defined by $\Omega= \frac{dz}{z} -  \eta$ where $\eta$ is the pull-back of a rational $1$-form on the basis. Integrability
of $\Omega$ implies that $\eta$ is closed and therefore $\Omega$ is also closed. The pull-back of $\Omega$  under the section $\sigma$ gives a closed rational $1$-form on $X$ defining $\mathcal F$.

Suppose that $\nabla$ is not logarithmic, and as in the proof of Theorem \ref{T:add} consider the decomposition
of $\nabla$ in a logarithmic connection $\nabla_{\log}$ and a closed rational $1$-form $h^* \beta_0$ without residues,
where $h: X \to \mathbb P^1$ is a morphism with irreducible general fiber.
Such a decomposition exists since $h^1(X,\mathbb C)=0$.

If every irreducible component $H$ of the singular divisor of $\nabla$ which dominates $\mathbb P^1$ (i.e. $h(H)=\mathbb P^1$) has integral
residues then  we claim that the representation factors through $h$. To verify this,  consider the residue divisor
of $\nabla$. It can be written in the form $\Res(\nabla) = V + T$ where $V$ is a $\mathbb C$-divisor supported on fibers
of $h$ and $T$ is a $\mathbb Z$-divisor with irreducible components transverse to $h$, i.e. not contained in fibers of $h$.
Consider the $\mathbb Q$-vector subspace $W$ of $H^2(X, \mathbb C)$ generated by the Chern classes of irreducible components of the support of $V$.
Notice that $W_{\mathbb C}$, the $\mathbb C$-vector subspace of $H^2(X, \mathbb C)$  generated by $W$, satisfies the identity
\[
W_{\mathbb C} \cap H^2(X,\mathbb Q) = W
\]
Since both $c_1(N\mathcal F)= c_1(V) + c_1(T)$ and $c_1(T)$  belong to $H^2(X,\mathbb Q)$, the Chern class of $V$ also belongs to $H^2(X,\mathbb Q)$ and consequently
belongs to $W$. Therefore there is a $\mathbb Q$-divisor $V'$ with support contained in the support of $V$ having the same Chern class as $V$.
Thus we can write $V = V' + V''$ where $V''$ is $\mathbb C$-divisor with zero Chern class supported
on fibers of $h$, and $V'$ is a $\mathbb Q$-divisor.
 Hodge index Theorem implies that $V''$ is a $\mathbb C$-linear combination of fibers of $h$. Using Proposition \ref{P:log}, we deduce that the
logarithmic connection $\nabla_{log}$ can be written as a sum of a logarithmic connection
$\nabla_{\log,\mathbb Q}$ with rational residues and $\eta_{log}$ the pull-back under $h$ of a logarithmic $1$-form on $\mathbb P^1$.
The monodromy representation of the connection $\nabla_{\log,\mathbb Q}$  is of finite order  around its poles, and since $h^1(X,\mathbb C)=0$, this suffices
to conclude that it is finite.
After passing to a finite ramified covering, we can assume that the monodromy of $\nabla_{\log,\mathbb Q}$  is trivial. We can then
argue as in the proof of Theorem~$\ref{T:add}$, using Picard-Fuchs equation  and
Proposition \ref{P:pushRiccati}, to conclude that $\mathcal F$ is the pull-back of a transversely affine Riccati foliation on a surface.

If there exists an irreducible component $H$  of the singular divisor of $\nabla$ with non integral residue generically transverse to the fibers of $h$,
then we are going to prove that $\mathcal F$ is defined by a closed rational $1$-form. The proof of  \cite[Lemma 9]{MR1185124}
shows that a codimension one foliation $\F$ on a projective manifold is defined by a rational closed $1$-form
if and only if the same holds true for the restriction of $\F$ to a sufficiently general hyperplane section. Therefore it suffices to consider the case where
$X$ is a surface.  We can further assume that $\mathcal F$ is a reduced foliation in the sense of Seidenberg.

Let us consider a general point $q$ at the intersection of $H$ (which is now a curve)  with a non-logarithmic pole of $\nabla$.
This point of intersection is a saddle node singularity for the foliation $\mathcal F$ which has formal normal form \cite[Section 1.1]{MR1726913}
\[
  \theta_{k,\mu} = x^{k+1} dy - y(1 + \mu x^k)dx .
\]
If it is not analytically conjugated to its formal normal form then \cite[Proposition 5.5]{MR1726913} implies that
every transverse affine structure for $\mathcal F$ has integral residues. Therefore, our assumptions implies that the saddle node
is analytically conjugated to its formal normal form. Hence  every  transverse affine structure for the germ of $\mathcal F$ at $q$
is defined by a pair  $({\theta_{k,\mu}}/{yx^{k+1}}, \lambda {\theta_{k,\mu}}/{yx^{k+1}})$, with $\lambda\in \C$. It follows that the sections  $\Sigma_1, \Sigma_2$ defined on the first
paragraph of the proof extend to meromorphic sections over a neighborhood of any component of the singular divisor which is generically transverse to the fibers of $h$. Adding to  $H$
some reduced and irreducible fibers of $h$ not contained in the singular divisor of $\nabla$ we obtain a reduced divisor with ample normal bundle, and
a neighborhood of it where the  sections $\Sigma_1$ and $\Sigma_2$
extend as meromorphic sections. We can apply \cite[Th\'eor\`eme 5]{MR0152674} or \cite[Corollary 6.8]{MR0232780} to extend these sections to the whole of $X$. We conclude that $\mathcal F$
in this case is defined by a closed rational $1$-form.
\end{proof}

\section{Proof of Corollary \ref{COR:B}}
Without loss of generality we can assume that $\omega$ has zero set of codimension at least two.
If $\omega$ is a polynomial Liouvillian integrable $1$-form on $\mathbb C^n$ without  invariant algebraic hypersurfaces
then $d \omega = dQ \wedge \omega$ for some polynomial $Q \in \mathbb C[x_1, \ldots, x_n]$. Let   $\mathcal F$ be the extension to $\mathbb P^n=\C^n\cup H_{\infty}$ of the codimension one foliation of $\mathbb C^n$ defined by $\omega$.

According to Theorem \ref{THM:A} there exists a rational map $F:\mathbb P^n \dashrightarrow S$, where $S$ is a $\mathbb P^1$-bundle over $\mathbb P^1$, and a transversely affine Riccati foliation $\mathcal R$ on $S$ such that $\mathcal F = F^*\mathcal R$. We will denote by $\pi: S \to \mathbb P^1$ the reference fibration of the Riccati foliation $\mathcal R$, i.e. $\mathcal R$ is everywhere transverse to the general fiber of $\pi$.

 By Stein factorization, we can that suppose the composition $\pi \circ F$ has irreducible general fiber.
 With a suitable choice of coordinates we can identify the restriction to $\mathbb C^n$ of $\pi \circ F$ with  a polynomial $P$ such that $P-c$ is irreducible for a general $c \in \mathbb C$ and $Q= A \circ P$ for some polynomial $A$ in one variable.

Up to birational transformations on $S$, we can also assume that
$S$ is a compactification of $\mathbb C^2$ such that the restriction of $\mathcal R$ to $\mathbb C^2$ has no invariant
algebraic curves and such that the divisor at infinity is $\mathcal R$ invariant and has no dicritical singularities.

The pre-image under $F$ of the divisor at infinity must be therefore invariant by $\mathcal F$. Since $\mathcal F$ has no algebraic invariant hypersurfaces on $\mathbb C^n$, it follows that this pre-image must be contained in the hyperplane at infinity. Hence $F$, in suitable coordinates, is nothing but a polynomial map $F:\mathbb C^n \to \mathbb C^2$. \qed

\section{Proof of Theorem \ref{THM:C}}
The statement of Theorem \ref{THM:C} is about the birational equivalence class of $\nabla$. Therefore we can assume that $\nabla$ is a meromorphic flat connection
on the trivial rank two vector bundle. Since it takes values in $\mathfrak{sl}(2)$, the connection matrix $\Omega$ has zero trace; and the reducibility hypothesis allows us to assume
that $\Omega$ is an upper triangular  matrix.
Therefore we can write
\[
\Omega = \left [ \begin{matrix} \eta/2&\omega\\0&-\eta/2\end{matrix} \right]
\]
for suitable rational $1$-forms $\omega$ and $\eta$. Since $\nabla$ is flat, the integrability equation $d\Omega+\Omega  \wedge \Omega=0$ holds true. This equation is equivalent to the
pair of equations $d\omega = \omega \wedge \eta$ and $d\eta=0$.

If $\omega$ is zero then there is nothing else to prove. Otherwise $\omega$ defines a transversely affine foliation $\F$ on $X$, with transversely affine structure given by the pair $(\omega,\eta)$.

Assume first that the foliation $\F$ has a non constant rational first integral $f \in \C(X)$.  Then $g\omega=df$ for some rational functions $f,g:X \dashrightarrow \mathbb P^1$.
We can replace $f$ by its Stein factorization, which still takes values in $\mathbb P^1$ since $h^1(X,\C)=0$.
If we apply the birational gauge transformation $\sqrt{g}\oplus1/\sqrt{g}$ on $Y$, the resolution of a double covering of $X$ determined by $\sqrt{g}$,
the connection matrix becomes
\[
\left [ \begin{matrix} \eta/2 - 1/2 d \log g & df\\0& -\eta/2 + 1/2 d \log g \end{matrix} \right]  =  \left [ \begin{matrix} hdf/2&df\\0&-hdf/2\end{matrix} \right] ,
 \]
 with $h \in \C(X)\subset \C(Y)$ satisfying $dh\wedge df=0$.  Since the general fiber of $f$ is irreducible, there exists $H \in \C(\mathbb P^1)$ such that $h = H\circ f$, and it becomes clear that the induced connection on $Y$ is birationally equivalent to the pull-back of a meromorphic connection on $\mathbb P^1$.

From now on we will assume that $\mathcal F$ does not admit a non constant rational  first integral. In particular, if $\F$ is defined by a closed rational $1$-form $\tilde{\omega}$ then
every other closed rational $1$-form defining $\F$ is a constant multiple of $\tilde{\omega}$.

If  there exists a generically finite Galois morphism $p:Y\to X$ such that $p^* \mathcal F$ is  defined by a closed rational $1$-form (case (1) of Theorem \ref{THM:A}) then after applying a
 gauge transformation of the form $\sqrt{g}\oplus1/\sqrt{g}$, with $g\in \C(Y)$, we can assume that $\omega$ is closed, and consequently $\eta=\lambda \omega$, for some  $\lambda \in \C$.
 If $\lambda=0$ there is nothing else to prove, otherwise  we apply  the gauge transformation with matrix
 \[
 \left [ \begin{matrix} 1&1/\lambda\\0&1\end{matrix} \right]
 \]
to  obtain a diagonal connection matrix and we have the result.

Now, we suppose we are not in case $(1)$ of Theorem \ref{THM:A}, in particular  $\F$   is a rational pull-back of a Riccati foliation $\mathcal H_0$ on a surface $S$ (case (2) of Theorem \ref{THM:A}).
Then there exists $f,g,h \in \mathbb C(X)$ and $\alpha, \beta$ rational $1$-forms on $\mathbb P^1$ such that
\[
   g\omega = dh + f^* \alpha + (f^* \beta)h  \, .
\]
After applying the  gauge transformation $\sqrt{g}\oplus1/\sqrt{g}$ we can assume that $g = 1$, thus by proposition $\ref{uniqueness}$, $\eta = f^* \beta$.
If we apply the gauge transformation given by the matrix
\[
 \left [ \begin{matrix} 1& h \\0&1\end{matrix} \right]
\]
we obtain the connection form
\[
 \left [ \begin{matrix} f^* \beta/2 &    f^* \alpha \\0&- f^* \beta / 2 \end{matrix} \right]
\]
which is clearly a pull-back from a curve. \qed

\bibliographystyle{amsalpha}

\end{document}